\documentclass[11pt]{amsart}

\usepackage{amsmath}
\usepackage{amssymb}
\usepackage{amsthm}
\usepackage{graphicx}
\usepackage{tikz}
\usepackage{xcolor}

\allowdisplaybreaks

\usepackage[margin=1.2in]{geometry}
\def\pr{\mathbb{P}}
\def\E{\mathbb{E}}
\def\var{\mathbb{Var}}

\def\vol{\textrm{vol}}

\def\E{\mathbb{E}}
\def\R{\mathbb{R}}

\def\eps{\varepsilon}

\def\cD{\mathcal{D}}

\def\1{\mathbf{1}}

\def\lam {\lambda}

\def\tce{t_c + \eps}
\def\tce2{t_c + \frac{\eps}{2}}

\def\bv{\mathbf{v}}
\def\bT{\mathbf{T}}
\def\bX{\mathbf{X}}

\def\var{\text{var}}

\def\vol{\text{\rm vol}}

\def\FV{\mathrm{FV}}

\newtheorem*{theorem*}{Theorem}
\newtheorem{theorem}{Theorem}
\newtheorem{lemma}[theorem]{Lemma}
\newtheorem{cor}[theorem]{Corollary}

\newtheorem*{defn*}{Definition}
\newtheorem{prop}[theorem]{Proposition}
\newtheorem*{prop*}{Proposition}

\newtheorem*{conj*}{Conjecture}

\newtheorem*{fact*}{Fact}

\begin{document}
\title[On the hard sphere model and sphere packings in high dimensions]{On the hard sphere model and sphere packings \\ in high dimensions}

\author{Matthew Jenssen}
\author{Felix Joos}
\author{Will Perkins}

\thanks{This research leading to these results was supported in part by EPSRC grants EP/P009913/1 and EP/M009408/1.}
\address{University of Oxford}
\email{matthew.jenssen@maths.ox.ac.uk}
\address{University of Birmingham}
\email{f.joos@bham.ac.uk}
\address{University of Birmingham}
\email{math@willperkins.org}
\date{\today}

\begin{abstract}
We prove a lower bound on the \textit{entropy} of sphere packings of $\R^d$ of density $\Theta(d \cdot 2^{-d})$.  The entropy measures how plentiful such packings are, and our result is significantly stronger than the trivial lower bound that can be obtained from the mere existence of a dense packing. 
Our method also provides a new, statistical-physics-based proof of the $\Omega(d  \cdot 2^{-d})$ lower bound on the maximum sphere packing density  by showing that the expected packing density of a random configuration from the hard sphere model is at least $(1+o_d(1)) \log(2/\sqrt{3}) d \cdot 2^{-d}$ when the ratio of the fugacity parameter to the volume covered by a single sphere is at least $3^{-d/2}$.  Such a bound on the  sphere packing density was first achieved by Rogers, with subsequent improvements to the leading constant by Davenport and Rogers, Ball, Vance, and Venkatesh.

\end{abstract}

\maketitle

\section{Sphere packings in high dimensions}
\label{sec:spherepack}

The \textit{sphere packing density} of $d$-dimensional Euclidean space, $\theta(d)$, is the supremum of the packing density over all packings $\mathcal P$ of $\R^d$ by equal-sized spheres;
that is,
\begin{align*}
\theta(d) &= \sup_{\mathcal P} \limsup_{R \to \infty}  \frac{ \vol(\mathcal P \cap B_R(0)) }{\vol(B_R(0))} \, , 
\end{align*}
where $B_R(x)$ is the closed ball of radius $R$ around $x$ and $\vol(\mathcal P \cap B_R(0))$ is the volume of $B_R(0)$ covered by spheres in the packing $\mathcal P$.
The precise value of $\theta(d)$ is known in only a small number of dimensions; to be precise for $d\in \{1,2,3,8,24\}$. 
While $d=1$ is trivial and $d=2$ is elementary but not trivial, the proof for $d=3$ was a monumental achievement of Hales~\cite{hales2005proof}, 
and the cases $d=8$ and $d=24$ were proved only very recently following a breakthrough of Viazovska~\cite{viazovska2017sphere} ($d=8$) and  Cohn, Kumar, Miller, Radchenko, and Viazovska~\cite{cohn2016sphere} ($d=24$); see~\cite{cohn2016conceptual} for an exposition of these recent developments.

Optimal sphere packings in  high dimensions are even more mysterious. 
It is not even clear whether lattice  packings achieve the optimal packing density or if the best packings are disordered.
A lower bound of $\theta(d) \ge 2^{-d}$ is trivial. Take any saturated packing; doubling the radii of the spheres must cover all of $\R^d$, or else another center could be added. Therefore the original density must be at least $2^{-d}$.  This bound has been improved by a factor of $d$ by Rogers~\cite{rogers1947existence}, with subsequent improvements to the constant by Rogers and Davenport~\cite{davenport1947hlawka}, Ball~\cite{ball1992lower}, Vance~\cite{vance2011improved} (in dimensions divisible by $4$), culminating in the bound of Venkatesh~\cite{venkatesh2012note} that $\theta(d) \ge (65963 +o_d(1)) d \cdot 2^{-d} $. Venkatesh also gains an additional $\log \log d$ factor in a sparse sequence of dimensions.  An upper bound of $\theta(d) \le 2^{- (.599\dots +o_d(1)) \cdot d}$ is due to Kabatiansky and Levenshtein~\cite{kabatiansky1978bounds}; Cohn and Zhao~\cite{cohn2014sphere} made a recent constant factor improvement.

Notably there has been no progress in closing the gap on an exponential scale between the trivial lower bound and the Kabatiansky and Levenshtein upper bound. See the books of Rogers~\cite{rogers1964packing}, Conway and Sloane~\cite{conway2013sphere}, and Cohn~\cite{cohn2016packing} for an overview of results and techniques in the area.

Several of the previous proofs of lower bounds on $\theta(d)$  analyze a random lattice packing by way of the Siegel mean-value theorem~\cite{siegel1945mean} or variants thereof; a bound of $2 \cdot 2^{-d}$ is achieved by analyzing a uniform random lattice (see Proposition~6.1 of~\cite{cohn2016packing}); by imposing additional symmetries on the random lattice  Vance~\cite{vance2011improved} and Venkatesh~\cite{venkatesh2012note} gain a factor $d$ and an improved constant.  But optimal packings in high dimensions are not necessarily lattice packings (see the conjectures of Torquato and Stillinger~\cite{torquato2006new}). If this is so, then we need different tools and constructions.  One natural candidate is the \textit{hard sphere model} from statistical physics (`hard' spheres since the only interaction between particles is the hard constraint that spheres cannot overlap).  This is a probability distribution over sphere packings governed by a fugacity parameter $\lam >0$. The larger $\lam$, the larger the typical density of a random packing from the model.

Here we utilize the hard sphere model to analyze sphere packings in high dimensions.  We show that for an appropriate choice of the fugacity, the expected packing density of a configuration drawn from the hard sphere model is $\Omega(d\cdot 2^{-d})$.  The argument not only gives a statistical physics proof of the lower bound on $\theta(d)$, but also gives a lower bound on the entropy of sphere packings of this density. We define the entropy precisely in Section~\ref{sec:hardspheres}, but it essentially expresses the exponential order of the fraction of sets of $\alpha n$ points in a ball of volume $n$ that are centers of a valid sphere packing in $\R^d$.  That is, it is a measure of how plentiful packings of a given density are. 

The proof technique is general; in fact a version of the argument in a discrete setting~\cite{davies2016average} 
(where the relevant statistical physics model is the hard-core model) 
states that a uniformly random independent set chosen from a triangle-free graph of maximum degree $r$ occupies at least a $\log r /r$ fraction of the vertices in expectation.  
This result gives an alternative proof of Shearer's bound of  Ramsey number $R(3,k) \le (1+o(1)) k^2/ \log k$~\cite{shearer1983note}, which is itself a sharpening of the independent set result of Ajtai, Koml\'os, and Szemer\'edi~\cite{ajtai1980note} used by Krivelevich, Litsyn, and Vardy~\cite{krivelevich2004lower} (following~\cite{jiang2004asymptotic}) to give an alternative proof of the $\Omega(d \cdot 2^{-d})$ lower bound on $\theta(d)$ by formulating the problem in terms of finding a large independent set in a graph derived by discretizing a region in $\R^d$. 
Since the first version of this paper, we have also used a variant of the method to prove lower bounds on the kissing number and size of spherical codes in high dimensions~\cite{JJP18kissing}.

In principle, the hard sphere model is a good random model with which to study optimal and near optimal sphere packings, as typical packings from the model will have density arbitrarily close to $\theta(d)$ for a large enough choice of the fugacity parameter $\lam$. Analyzing the typical packing density, however, is another matter, and we do not expect our particular technique, which relies only on local information, to improve the exponential order of the lower bound on $\theta(d)$.  
In the analogy with independent sets in graphs, the $\Omega(d \cdot 2^{-d})$ bound corresponds to the  $\Omega(\log r/r) $ lower bound on the independence ratio of a $r$-regular triangle-free graph.  
However, random $r$-regular graphs and random $r$-regular bipartite graphs have the same local structure asymptotically yet have drastically different  independence ratios: $2 \log r /r $ and $1/2$ respectively.

In Section~\ref{sec:hardspheres}, we explain the hard sphere model in detail and state our main result (Theorem~\ref{thmAlpaLB}).
In Section~\ref{sec:LB}, we prove Theorem~\ref{thmAlpaLB}.  
In Section~\ref{secLBpartition}, we use Theorem~\ref{thmAlpaLB} to prove a lower bound on the volume of sphere packings of density $\Theta(d \cdot 2^{-d})$.  This lower bound is significantly larger than the trivial bound obtained by shrinking the spheres of a dense packing and allowing the centers to move locally. 

In what follows $\log x$ always denotes the natural logarithm of $x$.  We use standard asymptotic notation with a subscript indicating the parameter with respect to which we take asymptotics.  So $f(d) = o_d(1)$ if $\lim_{d \to \infty} f(d) =0$. For $x,y\in \R^d$, we let $d(x,y)$ denote the Euclidean distance between $x$ and $y$, and for $X\subseteq\R^d$ we let $d(X,y)=\inf_{x\in X}d(x,y)$. The sphere of radius $r$ centered at $x$ in $\R^d$ is $\{ y: d(x,y)=r\}$, while the (open) ball of radius $r$ is $\{ y: d(x,y) < r \}$.

\section{The hard sphere model}
\label{sec:hardspheres}

 The \textit{hard sphere model} is a probability distribution over configurations of non-overlapping, identical spheres in a bounded subset of Euclidean space (that can be extended with a limiting argument to a distribution on packings of  all of $\R^d$).  
 There are two variants of the model: the \textit{canonical ensemble} is a uniformly random packing of a given fixed density and the \textit{grand canonical ensemble} is a random packing with variable density governed by a fugacity parameter $\lam>0$.    
 The hard sphere model is a simple model of a gas or fluid with no interactions apart from the hard constraint that molecules cannot overlap.  In dimension~$2$ and $3$ the model is expected to exhibit a \textit{freezing} phase transition, though proving this remains an open mathematical problem Such a phase transition would show that freezing and crystallization can be explained by purely geometric concerns.  The nature of such a phase transition may be different in $2$ dimensions than in $3$: Richthammer~\cite{richthammer2007translation} has proved that there can be no translational symmetry breaking in dimension $2$.   For more see L{\"o}wen's survey~\cite{lowen2000fun}. 

To define the model precisely, we assume the spheres of our packings have volume $1$ and  denote by $r_d$ the radius of a ball of volume $1$ in $\R^d$. For a bounded, measurable subset $S \subset \R^d$, let $C_k(S)$ be the set of unordered $k$-tuples of points from $S$; that is, 
\[ C_k(S)=\{ \{x_1, \dots x_k\}: x_i\in S\,  \forall\, i\} \, . \]
  Let 
  \[ P_k(S)= \{ \{ x_1, \dots x_k \} \in C_k(S): d(x_i,x_j) > 2 r_d \, \forall \, i \ne j \} \, ; \]
   that is, $P_k(S)$ is the subset of $C_k(S)$ consisting of the centers of packings of spheres of volume~$1$. Note that we allow centers near the boundary of $S$, so the spheres themselves need not lie entirely within $S$. 

The \textit{canonical hard sphere model} on $S$ with $k$ centers is simply a uniformly random  $k$-tuple $\mathbf X_k \in P_k(S)$.    The \textit{partition function} of the canonical hard sphere model on $S$ is the function
\begin{align}
\hat Z_S(k) &= \frac{1}{k!} \int_{S^k}  \mathbf 1_{ \cD(x_1,\ldots,x_k)}  \, dx_1 \cdots dx_k    \,,
\end{align}
where for $x_1,\ldots,x_k\in \R^d$,  the expression  $\cD(x_1,\ldots,x_k)$ denotes the event that 
$d(x_i,x_j)>2r_d$ for all distinct $i,j\in [k]$. In other words, $\hat Z_k(S)$ is the volume of $P_k(S)$ in the space of unordered $k$-tuples from $S$.  As the volume of $C_k(S)$  is $\vol(S)^k/k!$, the probability that $k$ uniformly random points in $S$ are the centers of a sphere packing is  $\frac{k!}{\vol(S)^k} \hat Z_S(k)$.

In the canonical ensemble the number of centers is fixed.  In the grand canonical ensemble we imagine $S$ lying in some larger region with which it can exchange particles, and so the number of centers is allowed to fluctuate. 
   
The \textit{grand canonical hard sphere model} on a bounded, measurable set $S \subset \R^d$ at fugacity $\lam$ is a random set $\mathbf X$ of unordered points, with $\mathbf X$ distributed according to a Poisson point process of intensity $\lam$ conditioned on the event that $d(x,y) > 2 r_d$ for all distinct $x,y \in \mathbf X$.

The partition function of the grand canonical hard sphere model on $S$ is
\begin{align}
\label{eqPartitionFunction}
Z_{S}(\lam)&=  \sum_{k \ge 0} \lam^k  \hat Z_S(k) 
\end{align}
where we take $\hat Z_S(0)=1$.   If $S$ is bounded then $Z_S(\lam)$ is a polynomial in $\lam$.

Note that the fugacity $\lam$ is not an absolute quantity: defining the model with spheres of a different size would lead to a different scaling of the fugacity.  The right absolute parameter to consider is the ratio of the fugacity to the volume enclosed by a single hard sphere; as we consider spheres of volume $1$ here, this ratio is $\lam$ as well.

In both the canonical and grand canonical ensembles, the partition function and its normalized logarithm play a central role in the study of the hard sphere model. Let $B_n = B_{n^{1/d} \cdot r_d}(0)$ be the ball of volume $n$ around the origin in $\R^d$. It follows from subadditivity that the limits
\begin{align*}
 f_d(\alpha)&:= \lim_{n \to \infty} \frac{1}{\alpha n} \log  \frac{\hat Z_{B_n}(\lfloor \alpha n \rfloor) }{n^{\lfloor \alpha n \rfloor}/(\lfloor \alpha n \rfloor)!}\\
 g_d(\lam)&:= \lim_{n \to \infty} \frac{1}{n} \log Z_{B_n}(\lam)
\end{align*}
 exist for $\alpha \in (0, \theta(d))$ and $\lam >0$.  We will call $f_d(\alpha)$ the \textit{entropy density} of sphere packings of $\R^d$ at density $\alpha$, and $g_d(\lam)$ the \textit{pressure} of the hard sphere model.  Both are measurements of how plentiful sphere packings are in $\R^d$. The entropy density is minus the thermodynamic \textit{free energy}, which itself is the large deviation rate function of the probability that $\alpha n$ random points in $B_n$ form a sphere packing. We dispense with the minus sign so that a lower bound on $f_d(\alpha)$ corresponds to a lower bound on the quantity of sphere packings.  Dividing by $\alpha n$ ensures that $f_d(\alpha)$ is independent of the choice of the size of spheres in our packings. See e.g.~\cite{radin2005structure} for a discussion of the entropy density in dimension $3$.  
 
The statistical physics definition of a phase transition in the hard sphere model is that the entropy density (resp.~the pressure)  is non-analytic at some $\alpha^* \in (0, \theta(d) )$ (resp.~at some $\lam^* >0$). See \cite{fernandez2007analyticity,hofer2015disagreement,pulvirenti2012cluster} for some recent results showing that the entropy density or pressure is analytic below some threshold in $\alpha$ or $\lambda$. See also~\cite{diaconis2011geometric,hayes2014lower,kannan2003rapid} for results showing that certain Markov chains for sampling from these models mix rapidly below a given threshold.  

  In fact in the large volume limit the two ensembles are essentially equivalent, as for each $\lam >0$, there is a typical density $\alpha(d,\lam)$ with small fluctuations.  However, computing this conversion function $\alpha(d,\lam)$ is as difficult as understanding both the sphere packing problem and the problem of phase transitions in the hard sphere model, as $\lim_{\lam \to \infty} \alpha(d,\lam) = \theta(d)$ (e.g.~\cite{mase2001packing}) and $\alpha(d,\lam)$ is non-analytic at $\lam$ at which $g_d(\lam)$ is non-analytic.  The main task of this work is to prove a lower bound on $\alpha(d,\lam)$. 

The \textit{expected packing density}, $\alpha_S(\lam)$, of the hard sphere model is simply the expected number of centers in $S$ normalized by the volume of $S$; 
that is, 
\[ \alpha_S(\lam) = \frac{ \E_{S,\lam} |\mathbf X |}{\vol(S)}  .\]
Here and in what follows the notation $\pr_{S,\lam}$ and $\E_{S,\lam}$ indicates probabilities and expectations with respect to the grand canonical hard sphere model on a region $S$ at fugacity $\lam$. We may omit the subscripts if $S$ and $\lam$ are clear from the context.

The expected packing density can be expressed as the derivative of the normalized log partition function. 
We calculate
\begin{align}\notag
\alpha_S(\lam) &= \frac{1}{\vol(S)} \sum_{k=1}^\infty k \cdot \pr_{S,\lam}[|\mathbf X|=k] \\
&=\frac{1}{\vol(S)  }  \notag \sum_{k=1}^\infty \frac{k \cdot \lam^k \hat Z_S(k)}{Z_S(\lam)}    \\
&= \frac{1}{\vol(S)}  \frac{\lam \cdot Z_S'(\lam) }{Z_S(\lam)    }\label{eq:aZ}  \\
&= \frac{\lam }{\vol(S)}  ( \log Z_S(\lam))' \, .\label{eq:alogZ}
\end{align}

The next lemma shows that the expected packing density of the hard sphere model provides a lower bound for $\theta(d)$.
\begin{lemma}
The asymptotic expected packing density of $B_n \subset \R^d$  is a lower bound on the maximum sphere packing density. That is, for any $\lam>0$,
\begin{equation*}
 \theta(d) \ge \limsup_{n \to \infty}  \alpha_{B_n}(\lam)\, .
 \end{equation*}
\end{lemma}
\begin{proof}
 First note that 
\begin{align}
\label{eqthetaSup}
\theta(d) &= \limsup_{n \to \infty} \sup_{ X \in \mathcal P(B_n,r_d)}  \frac{|X|}{n}
\end{align}
where $\mathcal P(B_n,r_d)$ is the set of all packings of $B_n$ by spheres of radius $r_d$ (where again only the centers need be in $B_n$); that is, sets of distinct points $X \subset B_n$ so that $d(x_i,x_j) >2 r_d$ for all distinct $x_i,x_j \in X$. 
The equality~\eqref{eqthetaSup} relies on the fact that volume of a ball in $\R^d$ grows subexponentially fast as a function of its radius, and so deleting centers from the boundary of $B_n$ has a negligible effect on the packing density as $n \to \infty$.  Now from the definition of the expected packing density, $\sup_{ X \in \mathcal P(B_n,r_d)}  \frac{|X|}{n} \ge  \alpha_{B_n}(\lam)$ for any $\lam$.
\end{proof}

Our main  result is the following lower bound on the expected packing density.
\begin{theorem}
\label{thmAlpaLB}
Let $S\subset \R^d$ be bounded,  measurable, and  of positive volume. 
Then for any $\lam \ge  3^{-d/2}$, we have
\begin{align*}
\alpha_{S}(\lam) &\ge (1+o_d(1)) \frac{\log (2/\sqrt{3}) \cdot d  }{ 2^d} \,.
\end{align*}
\end{theorem}

As a corollary, by applying Theorem~\ref{thmAlpaLB} to $B_n$, we obtain the following lower bound on the sphere packing density of the $d$-dimensional Euclidean space.

\begin{cor}
\[
 \theta(d)\ge(1+o_d(1)) \frac{\log (2/\sqrt{3}) \cdot d  }{ 2^d}\, .
\]
\end{cor}

The fact that we achieve the bound in Theorem~\ref{thmAlpaLB} for $\lam$ as small as $3^{-d/2}$ has no implication on the bound obtained on $\theta(d)$, but it allows us to prove  non-trivial lower bounds on the entropy density and pressure. 

\begin{theorem}
\label{thmFElb}
For all $\lam = e^{-cd}$ with $c \in  [\frac{\log 3}{2}, \log 2)$,  
\begin{align*}
g_d(\lam) &\ge \left(\frac{(\log 2 - c)^2}{2}+o_d(1) \right ) \cdot \frac{ d^2}{2^d} \, .
\end{align*}
\end{theorem}

\begin{theorem}
\label{thmCanonBound}
There exists $\alpha = \alpha(d) = (1+o_d(1))  \frac{ \log(2/\sqrt{3}) \cdot d }{2^d} $ so that 
\begin{align*}
f_d(\alpha) &\ge -   (1+o_d(1)) \log(2/\sqrt{3}) \cdot d   \, .
\end{align*}
\end{theorem}

The lower bound in Theorem~\ref{thmCanonBound} matches, up to a factor $2$, a formula for the entropy of hard spheres conjectured in the physics literature to hold for densities up to either the crystallization phase transition or the glass transition, whichever comes first~\cite{frisch, Parisi} (see also~\cite{GlassJamming} for an overview of the mean-field approach to hard spheres). 

 Of course even the existence of a sphere packing of density $\Theta(d \cdot 2^{-d})$ implies some positive volume of sphere packings at a slightly lower density by shrinking the spheres and allowing their centers to move locally.  Such a lower bound on the canonical partition function is called the `cell model' lower bound in statistical physics (see e.g.~\cite{lowen2000fun}, Section 4.2). While the cell model is a rigorous lower bound on $\hat Z_S(k)$ at all densities, it is thought to be approximately accurate if the model is in a crystalline phase.
In Section~\ref{sec:compare} we compare the bound from Theorem~\ref{thmCanonBound} to this cell model lower bound, and show that is is significantly stronger.

In fact, to achieve the bound in Theorem~\ref{thmCanonBound} through the existence of a dense packing and the cell model lower bound would require $\theta(d) \ge (2- \eps)^{-d}$ for some $\eps>0$. So in a sense we can say that either there is no crystallization at density $\Theta(d \cdot 2^{-d})$ or there are exponentially better sphere packings than currently known.  We leave precise statements to this effect for future work, but conclude by observing that these two challenging problems in geometry and statistical physics, determining the asymptotic sphere packing density and determining whether or not the hard sphere model exhibits a phase transition, closely complement one another and understanding their relationship may open the way to further progress in both areas.

\section{A lower bound on the expected packing density}
\label{sec:LB}

In this section we prove Theorem~\ref{thmAlpaLB}.  We start with some useful identities and inequalities.

When $\lambda$ is  large, the model favors configurations with more spheres.
It is a standard  fact that $\alpha_S(\lam)$ is strictly increasing in $\lambda$.
\begin{lemma}
\label{lemMonotone}
Let $S\subset \R^d$ be bounded,  measurable, and  of positive volume.
Then the expected packing density $\alpha_S(\lam)$ is a strictly increasing function of $\lam$. 
\end{lemma}
\begin{proof}
We use~\eqref{eq:alogZ} and calculate 
\begin{align*}
\lam \cdot  \vol(S) \cdot \alpha_S'(\lam)&= \lam^2 ( \log Z_S(\lam))'' +  \lam ( \log Z_S(\lam))' \\
&= \lam^2 \cdot \frac{Z_S(\lam)  Z_S''(\lam) - (Z_S'(\lam))^2   } {Z_S^2(\lam)} + \frac{ \lam Z_S'(\lam)}{Z_S(\lam)} \\
&= \E_{S,\lam}[ |\mathbf X|(|\mathbf X|-1)] - (\E_{S,\lam}[|\mathbf X|]  )^2 + \E_{S,\lam}[ |\mathbf X|] \\
&= \var_{S,\lam} [ |\mathbf X|] > 0 \,,
\end{align*}
and so $\alpha_S(\lam)$ is a strictly increasing.
\end{proof}

Let $\FV_S(\lam)$ denote the expected \emph{free volume} of the hard sphere model; 
that is, 
the expected fraction of the  volume of $S$ containing points that are at distance at least $2 r_d$ from the nearest center; or in other words, the expected fraction of volume at which a new sphere could be legally placed.
A key fact in our argument is the following link between $\alpha_S(\lam)$ and $\FV_S(\lam)$.
\begin{lemma}\label{lem:alphaFV}
Let $S\subset \R^d$ be bounded,  measurable, and  of positive volume.
Then
\[ \alpha_S(\lam) = \lam \cdot \FV_S(\lam). \]
\end{lemma}
\begin{proof}
We simply use the definition of $\alpha_S(\lam)$ and compute
\begin{align*}
\alpha_S(\lam)
&=\frac{\E_{S,\lam} |\bX |}{\vol(S)}\\
&= \frac{1}{\vol(S)}\sum_{k=0}^\infty (k+1) \pr_{S,\lam}[|\bX|=k+1] \\
&= \frac{1}{\vol(S)Z_S(\lam)}\sum_{k=0}^\infty  \int_{S^{k+1}} \frac{\lam^{k+1}}{k!}   \1_{\cD(x_0,\ldots,x_k)}  \, dx_1 \cdots dx_kdx_0   \\
&= \frac{\lam}{\vol(S)Z_S(\lam)}  \int_S \left [ 1+\sum_{k=1}^\infty  \int_{S^k} \frac{\lam^{k}}{k!} \1_{\cD(x_0,\ldots,x_k)}   \, dx_1 \cdots dx_k  \right ]\, dx_0  \\
&= \lam \cdot \FV_S(\lam).
\end{align*}
\end{proof}

Now consider the following two-part experiment: sample a configuration of centers $\mathbf X$ from the hard sphere model on $S$ at fugacity $\lam$ and independently choose a point $\mathbf v$ uniformly from~$S$. 
We define the random set 
\[ \mathbf T = \{ x \in B_{2r_d}(\mathbf v)\cap S : d(x,y) > 2r_d \,\, \forall  \, y\in \mathbf X\cap B_{2r_d}(\bv)^c \} . \] 
That is, $\mathbf T$ is the set of all points of $S$ in the $2r_d$ ball around $\mathbf v$ that are not blocked from being a center by a center outside the $2r_d$ ball around $\mathbf v$. 
We call $\mathbf T$ the set of externally uncovered points in the neighborhood of $\bv$, in analogy with the terminology used in~\cite{Davies2015,davies2016average} in the discrete case.
Note that $\mathbf T$ depends only on $\mathbf X \cap B_{ 2 r_d}(\mathbf v)^c$ -- the presence or absence of centers inside $B_{ 2 r_d}(\mathbf v)$ has no effect on $\mathbf T$ (see Figure~\ref{figExternal}). 

\begin{figure}[h]
\centering
\begin{tikzpicture}[scale=.8]

\draw[fill,black!40] (0,0) circle (2);
\draw[thick] (0,0) circle (0.05);
\draw (0,0) node[anchor=north east] {$\bv$};

\draw[thick,fill] (2,1.5) circle (0.06);
\draw[thick,fill] (3,-.5) circle (0.06);
\draw[thick,fill] (1,-2.5) circle (0.06);
\draw[thick,fill] (-2.5,1.1) circle (0.06);
\draw[thick,fill] (-1,-1) circle (0.06);

\begin{scope}
\clip (0,0) circle (2);

\draw[thick,fill,black!20] (2,1.5) circle (2);
\draw[thick,fill,black!20] (3,-.5) circle (2);
\draw[thick,fill,black!20] (1,-2.5) circle (2);
\draw[thick,fill,black!20] (-2.5,1.1) circle (2);

\end{scope}

\draw[very thick] (2,1.5) circle (2);
\draw[very thick] (3,-.5) circle (2);
\draw[very thick] (1,-2.5) circle (2);
\draw[very thick] (-2.5,1.1) circle (2);

\draw[very thick, dashed] 
(2,1.5) circle (1)
(3,-.5) circle (1)
(1,-2.5) circle (1)
(-2.5,1.1) circle (1);

\draw[very thick] (0,0) circle (2);

\draw[thick,dashed] (2,1.5) -- node[very near end, right ]{$2r_d$} +(120:2cm);

\end{tikzpicture}
\caption{An illustration of the set of externally uncovered points in the neighborhood of $\bv$ (shaded dark gray). The dashed circles represent the hard spheres which do not overlap.}
\label{figExternal}
\end{figure}
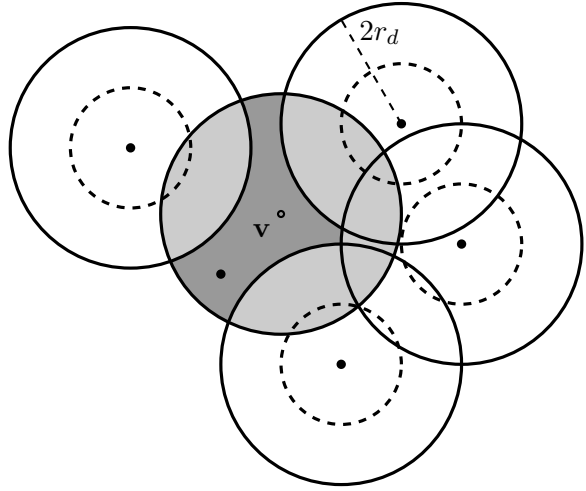

Since $\bX$ is a finite set of points it is clear that there exists some $\eps>0$ (depending on $\bX$) such that $B_\eps(\bv)\cap S\subseteq \bT$. If $S$ has positive volume then it is not difficult to show that $B_\eps(\bv)\cap S$ has positive volume almost surely and hence that $\vol (\bT)>0$ almost surely.

\begin{prop}
\label{propAlphas}
Let $S\subset \R^d$ be bounded,  measurable, and  of positive volume.
Then
\begin{align}
\label{eqalp1}
\alpha_S(\lam) &= \lam \cdot \E \left[ \frac{1}{Z_{\bT}(\lam)}  \right ]\,
\end{align}   
and
\begin{align}
\label{eqalp2}
\alpha_{S}(\lam) &\ge 2^{-d}  \cdot \E \left [ \frac{\lam \cdot Z_{\bT}^\prime(\lam)}{Z_{\bT}(\lam)} \right ] \, ,
\end{align}
where both expectations are with respect to the random set $\bT$ generated by the two-part experiment defined above. 
\end{prop}

\begin{proof}
We use Lemma~\ref{lem:alphaFV} to conclude that
\begin{align*}
\alpha_S(\lam) &= \lam \cdot \FV_S(\lam) \\
&= \frac{\lam }{\vol (S)} \int_S \pr [d(\mathbf X,v) > 2 r_d]  \, dv  \\
&= \lam \cdot \E \left[\mathbf 1_{\bT \cap \mathbf X = \emptyset}     \right ]  \\
&=  \lam \cdot \E \left[\frac{1}{Z_{\bT}(\lam)}     \right ]  \, ,
\end{align*}
which gives~\eqref{eqalp1}.  
The last equality uses the spatial Markov property of the hard sphere model: 
conditioned on $\mathbf X \cap B_{ 2 r_d}( \mathbf v)^c$, the distribution of $\mathbf X \cap  B_{ 2 r_d}( \mathbf v)$ is exactly that of the hard sphere model on the set $\bT$. 

Next, using a double counting argument and the fact that $\vol( S \cap B(v, 2 r_d)) \le 2^d$ for any $v \in S$, we obtain
\begin{eqnarray*}
\alpha_{S}(\lam) 
&\ge& 2^{-d} \cdot \E[| \bX \cap B_{ 2 r_d}(\bv)|] \\
&=& 2^{-d} \cdot \E [\alpha_{\bT}(\lam)\cdot \vol(\bT) ]\\
&\stackrel{(\ref{eq:aZ})}{=}& 2^{-d}  \cdot \E \left [ \frac{\lam \cdot Z_{\bT}^\prime(\lam)}{Z_{\bT}(\lam)} \right ] \, .
\end{eqnarray*}
\end{proof}

\begin{prop}
\label{propLB1s}
Let $S \subset \R^d$ be bounded and measurable. Then
\begin{align}
\label{eqZlb1}
\log Z_S(\lam) & \le \lam \cdot \vol(S)\,
\end{align}
and if in addition $S$ is of positive volume, then
\begin{align}
\label{eqAlpLb1}
\alpha_S(\lam) &\ge \lam \cdot e^{- \lam \cdot \E  [\vol(\bT) ]} \, .
\end{align}
\end{prop}
\begin{proof}
From~\eqref{eqPartitionFunction}, the definition of $Z_S(\lam)$,
we have $Z_S(\lam) \le \sum_{k= 0}^\infty \frac{\lam^k}{k!} \cdot \vol(S)^k = e^{\lam \cdot \vol(S)}$.
Turning to~\eqref{eqAlpLb1}, we conclude	
\begin{eqnarray*}
\alpha_S(\lam) 
&\stackrel{(\ref{eqalp1})}{=}& \lam \cdot \E \left[ \frac{1}{Z_{\mathbf T}(\lam)}  \right ] \\
&\stackrel{(\ref{eqZlb1})}{\ge}&  \lam \cdot \E \left[e^{-\lam \cdot \vol(\bT)}  \right ]  \\
&\ge& \lam \cdot e^{- \lam \cdot \E  [\vol(\bT) ]}, 
\end{eqnarray*}
where the last inequality is an application of Jensen's Inequality.
\end{proof}

\begin{lemma}
\label{lemmaT}
Let $S \subseteq B_{ 2 r_d}(0)$ be measurable. Then
\begin{align}\label{eq:3d2}
	\E[\vol(B_{ 2 r_d}(\mathbf u ) \cap S) ] \le 2 \cdot 3^{d/2} \,,
\end{align}
where $\mathbf u$ is a uniformly chosen point in $S$. In particular
\begin{align}\label{eq:uts}
	\alpha_S(\lam) \ge \lam\cdot e^{- \lam \cdot 2 \cdot 3^{d/2}} \,.
\end{align}
\end{lemma}
The geometric fact \eqref{eq:3d2} is related to the fact used in~\cite{krivelevich2004lower}; here we consider the volume of the intersection of a sphere with an arbitrary set, but we bound this by the intersecting volume of two identical spheres, as in~\cite{krivelevich2004lower}.     
\begin{proof}[Proof of Lemma~\ref{lemmaT}]
Clearly, we may assume that $S$ has positive volume.
We  write
\begin{align*}
\E  [\vol(B_{ 2 r_d}(\mathbf u ) \cap S) ] &= \frac{1}{\vol(S)} \int_{S} \int_{S} \mathbf 1_{d(u,v) \le 2r_d}   \, dv \, du \\
&= \frac{2}{\vol(S)} \int_{S} \int_{S} \mathbf 1_{d(u,v) \le 2r_d} \cdot \mathbf 1_{\|v\| \le \|u \|}  \, dv \, du  \\
&\le 2 \max_{u \in B_{ 2 r_d}(0)} \int_S  \mathbf 1_{d(u,v) \le 2r_d} \cdot \mathbf 1_{\|v\| \le \|u \|}   \, dv  \\
&\le 2 \max_{u \in B_{ 2 r_d}(0)}  \vol \left ( B_{ 2 r_d}(u) \cap B_{\|u\|}(0)   \right) \, .
\end{align*}
Now suppose the point $u$ is at distance $t r_d$ from $0$ for some $t\in [0,2]$. 
 We may assume  that $t \ge \sqrt{2}$ as otherwise $\vol(B_{\|u\|}(0))\le 2^{d/2}$.  
 Then, 
by bounding the volume of the intersection of two balls by the volume of a containing ball (see Figure~\ref{FigIntersect}), 
we have
\begin{align*}
\vol\left(B_{ 2 r_d}(u) \cap B_{ t r_d}(0)   \right) 
&\le  \vol (B_{ 2 r_d \sqrt{1-t^{-2}}}(0 )   ) \\
&\le \left(2 \sqrt{1-t^{-2}}  \right)^d  \, ,
\end{align*}
and so
\begin{align*}
\E [\vol(B_{2r_d }(\mathbf u) \cap S) ] 
&\le \max \left\{2^{d/2},2\cdot \max_{\sqrt{2} \le t \le 2}  \left(2 \sqrt{1-t^{-2}}  \right)^d\right\} \\
&= 2 \cdot 3^{d/2} \, .
\end{align*}
This establishes \eqref{eq:3d2}. It follows that $\E[\vol(\bT)]\leq 2 \cdot 3^{d/2}$ and so \eqref{eq:uts} follows from \eqref{eqAlpLb1}.
\begin{figure}[t]
\centering
\begin{tikzpicture}[scale=1.3]

\draw [line width=1.2pt] (0,0) circle (1.75);
\draw[thick,fill] (0,0) circle (0.05);
\draw (0,0) node[anchor=north] {$0$};

\draw [line width=1.2pt] (1.75,0) circle (2);
\draw[thick,fill] (1.75,0) circle (0.05);
\draw (1.75,0) node[anchor=north west] {$u$};

\draw[thick,dashed, fill=black,fill opacity=0.100] (0.607,0) circle (1.64);
\draw[thick,fill] (0.607,0) circle (0.05);
\draw[thick,fill] (0.607,-1.64) circle (0.05);
\draw[thick,fill] (0.607,1.64) circle (0.05);

\draw[ultra thick,dashed] 
(0.607,1.64)--(0,0)-- (1.75,0)
(0.607,1.64)--  (1.75,0) ;
\draw[thick, dashed](0.607,-1.64)-- (0.607,1.64);
\draw (1.05,0.7) node {$2r_d$};
\draw (0.41,0.4) node {$tr_d$};
\draw (1.1,-0.2) node {$tr_d$};

\end{tikzpicture}
\caption{}
\label{FigIntersect}
\end{figure}
\end{proof}

Using these results we now prove   Theorem~\ref{thmAlpaLB}.

\begin{proof}[Proof of Theorem~\ref{thmAlpaLB}]
Let $S\subset \R^d$ be bounded,  measurable, and  of positive volume. 
Let $\alpha = \alpha_{S}(\lam)$.   
Then by Jensen's Inequality we obtain
\begin{align*}
\alpha &\stackrel{(\ref{eqalp1})}{=} \lam \cdot \E \left[ \frac{1}{Z_{\mathbf T}(\lam)}  \right ] \ge \lam \cdot e^{-   \E \log Z_{\mathbf T}(\lam)} \,,
\end{align*}
where as above the expectation is with respect to the two part experiment in forming the random set $\bT$.

On the other hand we have
\begin{eqnarray*}
\alpha &\stackrel{(\ref{eqalp2})}{\ge}& 2^{-d} \cdot \E \left [ \frac{\lam \cdot Z_{\bT}^\prime(\lam)}{Z_{\bT}(\lam)} \right ]  \\
&\stackrel{(\ref{eq:aZ})}{=}& 2^{-d} \cdot \E \left[ \vol(\bT) \cdot \alpha_{ \bT} (\lam) \right ] \\
&\stackrel{(\ref{eq:uts})}{\ge}& 2^{-d} \cdot \E \left[\lam \cdot \vol(\bT) \cdot e^{- \lam \cdot 2 \cdot 3^{d/2}}  \right ]  \\
&\stackrel{(\ref{eqZlb1})}{\ge} & 2^{-d}\cdot \E \left[  \log Z_{\bT}(\lam) \cdot e^{- \lam \cdot 2 \cdot 3^{d/2}}  \right ]    \\
&=&  2^{-d}\cdot e^{- \lam \cdot 2 \cdot 3^{d/2}}  \E [ \log Z_{\bT}(\lam)] \, .
\end{eqnarray*}
Combining these two lower bounds, and letting $z =  \E  \log Z_{\mathbf T}(\lam)$, we see that 
\begin{align*}
\alpha \ge \inf_z \max \left \{ \lam e^{- z}, z \cdot  2^{-d} e^{ -\lam \cdot 2 \cdot 3^{d/2}}   \right \} .
\end{align*}
Since $\lam e^{- z}$ is decreasing in $z$ and $ z \cdot  2^{-d} e^{ -\lam \cdot 2 \cdot 3^{d/2}}$ increasing, the infimum over $z$ of the maximum of the two expressions occurs when they are equal, that is, $\alpha \ge \lam e^{- z^*} $, where $z^*$ is the solution to 
\begin{align*}
\lam e^{-z} &=  z \cdot  2^{-d} e^{ -\lam \cdot 2 \cdot 3^{d/2}} \, ,
\end{align*}
or in other words,
\begin{align*}
z^* &=  W \left( \lam 2^d e^{ \lam \cdot 2 \cdot 3^{d/2}}  \right )
\end{align*}
where $W(\cdot)$ is the Lambert-W function. 
 Now take $\lam = d^{-1} 3^{-d/2}$ (in fact $\lam = \eps 3^{-d/2}$ for any $\eps = \eps(d)$ such that $\eps \to 0$ and $-\log(\eps)/d \to 0$ as $d \to \infty$ would suffice).  Recall that $W(x) = \log x - \log \log x + o(1)$ as $x \to \infty$.  This gives
 \begin{align*}
z^*&= W( \lam \cdot 2^d \cdot e^{2/ d}) \\
 &=\log \lam + d \log 2  - \log d - \log \log (2/\sqrt{3})+ o_d(1) 
 \end{align*}
and so
\begin{align*}
\alpha&\ge \lam e^{-z^*} = (1+o_d(1)) \frac{\log (2/\sqrt{3}) \cdot d  }{ 2^d} 
\end{align*}
which completes the proof of Theorem~\ref{thmAlpaLB}.
\end{proof}
Note that in the proof if we take $\lam = e^{-cd}$ for $c \in (\frac{\log 3}{2}, \log 2)$, then we obtain the following bound
\begin{align}
\label{eqalphaGenbound}
\alpha_{S}(\lam) &\ge (1+o_d(1)) \frac{ ( \log2 - c )  \cdot d}{2^d} \, .
\end{align}

\section{A lower bound on the entropy density and pressure}
\label{secLBpartition}

We first consider the grand canonical model and the pressure of the hard sphere model.  As shown in~\eqref{eq:alogZ},
the expected packing density is the scaled derivative of the log partition function;
that is $\alpha_S(\lam) =\frac{\lam }{\vol(S)}  ( \log Z_S(\lam))' \, .$
Theorem~\ref{thmAlpaLB} and inequality~\eqref{eqalphaGenbound} give a lower bound on the expected packing density; by integrating this bound we obtain the lower bound on the pressure stated in Theorem~\ref{thmFElb}.

\begin{proof}[Proof of Theorem~\ref{thmFElb}]
We compute
\begin{align*}
 \frac{1}{n} \log Z_{B_n} (\lam) &= \int_{0}^{\lam} \frac{1 }{n}  ( \log Z_{B_n}(t))'  \, dt \\
 &= \int_0^\lam \frac{\alpha_{B_n}(t)}{t} \, dt \\ 
&\ge -d \int_{\log 2}^{c} \alpha_{B_n}(e^{-ud})  \, du \\
&\ge (1+o_d(1)) \frac{d^2}{2^d}  \int^{\log 2}_{c} (\log2 - u)  \, dc \\
&=  \left(\frac{(\log 2 - c)^2}{2}+o_d(1) \right ) \cdot \frac{ d^2}{2^d} \,,
\end{align*}
and taking $n \to \infty$ gives the theorem. 
\end{proof}

Now recall the definition of the entropy density of sphere packings of $\R^d$ at density $\alpha$:
\begin{align*}
 f_d(\alpha)&= \lim_{n \to \infty} \frac{1}{\alpha n} \log  \frac{\hat Z_{B_n}(\lfloor \alpha n \rfloor) }{n^{\lfloor \alpha n \rfloor}/(\lfloor \alpha n \rfloor)!} \,.
\end{align*}
The entropy density is a measure of how plentiful sphere packings of a given density are, as it tells us, on a logarithmic scale, what fraction of point sets of a given density in a large region of $\R^d$ are the centers of a sphere packing.  We use Theorem~\ref{thmAlpaLB} to provide the lower bound on $f_d(\alpha)$ given in Theorem~\ref{thmCanonBound}. First let us record the  simple fact  that as sphere packings become more dense they become less plentiful. 

\begin{lemma}\label{lem:fdec}
$f_d(\alpha)$ is decreasing in $\alpha$.
\end{lemma}
\begin{proof}
Suppose $0 < \alpha < \alpha' < \theta(d)$.  
Since the limit
\begin{align*}
 f_d(\alpha)&= \lim_{n \to \infty} \frac{1}{\alpha n} \log  \frac{\hat Z_{B_n}(\lfloor \alpha n \rfloor) }{n^{\lfloor \alpha n \rfloor}/(\lfloor \alpha n \rfloor)!} \,
\end{align*}
exists it is enough to show
\begin{align*}
 \frac{1}{\alpha n} \log  \frac{\hat Z_{B_n}(\lfloor \alpha n \rfloor) }{n^{\lfloor \alpha n \rfloor}/(\lfloor \alpha n \rfloor)!} \ge  \frac{1}{\alpha' n'} \log  \frac{\hat Z_{B_{n'}}(\lfloor \alpha' n' \rfloor) }{n'^{\lfloor \alpha' n' \rfloor}/(\lfloor \alpha' n' \rfloor)!}
\end{align*}
for some sequence $n, n' \to \infty$.  Choose $n$ arbitrarily and set $n' = \frac{\alpha}{\alpha'} n$. Let $\tilde k = \alpha n = \alpha' n'$ and $k= \lfloor \alpha n \rfloor = \lfloor \alpha' n' \rfloor$.  Then we must show 
\begin{align*}
 \frac{1}{\tilde k} \log  \frac{\hat Z_{B_n}(k) }{n^{k}/k!} \ge  \frac{1}{\tilde k} \log  \frac{\hat Z_{B_{n'}}(k) }{n'^{k}/k!} \, ,
\end{align*}
or equivalently,
\begin{align*}
\frac{\hat Z_{B_n}(k) }{n^{k}} \ge \frac{\hat Z_{B_{n'}}(k) }{n'^{k}} \, .
\end{align*}
In words this is the statement that the probability $k$ uniform and independent random points in a ball of volume $n$ form a packing of balls of volume $1$ is at least the same probability in a ball of volume $n'$ with $n' < n$.  This follows from a simple scaling and coupling: it is the same as the statement that the probability $k$ uniform and independent random points in a ball of volume $n$ form a packing of balls of volume $1$ is at least the probability $k$ uniform and independent random points in a ball of volume $n$ form a packing of balls of volume $v$ with $v > 1$, and clearly the second event is contained in the first.  
\end{proof}

Now we prove our lower bound on the entropy density.

\begin{proof}[Proof of Theorem~\ref{thmCanonBound}]
Fix $d$ and $n$ sufficiently large. Choose $\lam\in[3^{-d/2}, 2\cdot 3^{-d/2}]$ so that $\var_{B_n, \lam} | \mathbf X | \le n^{3/2}$; such a $\lam$ always exists because otherwise, by the calculation of Lemma~\ref{lemMonotone}, we would have
\[
\alpha_{B_n}(2\cdot 3^{-d/2})=\frac{1}{n}\int_{0}^{2\cdot 3^{-d/2}}\frac{\var_{B_n, t} | \mathbf X |}{t}\, dt\ge\frac{1}{n}\int_{3^{-d/2}}^{2\cdot 3^{-d/2}}\frac{n^{3/2}}{t}\, dt=n^{1/2}\cdot \log 2>1\, .
\]
Note that for this choice of $\lam$ we have 
\begin{align}\label{eq:ELB}
\E_{B_n,\lam} |\mathbf X| \ge (1+o_d(1)) \frac{ \log(2/\sqrt{3}) \cdot d }{2^d}\cdot n
\end{align}
 by Theorem~\ref{thmAlpaLB}. By our bound on the variance and Chebyshev's inequality it follows that
\[
\pr_{B_n, \lam}\big [|\mathbf X|\in (\E_{B_n, \lam} |\mathbf X|- n^{4/5}, \E_{B_n, \lam} |\mathbf X|+ n^{4/5})\big]\ge1-\frac{1}{n^{1/10}}\, .
\]
Since there are at most $\lfloor2 n^{4/5}\rfloor$ integers in the interval $(\E_{B_n, \lam} |\mathbf X|- n^{4/5}, \E_{B_n, \lam} |\mathbf X|+ n^{4/5})$ we may pick some $k$ in this interval so that 
\[
\pr_{B_n, \lam}[|\mathbf X|=k]
=\frac{\lam^k \hat Z_{B_n}(k)}{Z_{B_n}(\lam)}\ge\frac{1-n^{-1/10}}{\lfloor2 n^{4/5}\rfloor}\ge\frac{1}{n}\, .\]
It follows that
\begin{align}\label{eq:grand}
\hat Z_{B_n}(k) &\ge \frac{1}{n} \frac{1}{\lam^k} Z_{B_n}(\lam) \ge \frac{1}{n} \frac{1}{\lam^k} \,,
\end{align}
where we used the trivial bound $Z_{B_n}(\lam) \ge1$.
Let $\alpha = k/n$ and note that by \eqref{eq:ELB} and our choice of $k$ we have $\alpha\ge(1+o_{n,d}(1)) \frac{ \log(2/\sqrt{3}) \cdot d }{2^d}$. It then follows from \eqref{eq:grand} that
\begin{align*}
 \frac{1}{\alpha n} \log \frac{\hat Z_{B_n}(k)}{n^{k}/k!} &\ge  \log \alpha   -   \log \lam -1 + o_n(1) \\
 &\ge - (1+o_{n,d}(1))\log(2/\sqrt{3}) \cdot d    \, .
\end{align*}
Taking $n \to \infty$ and recalling Lemma~\ref{lem:fdec} proves the theorem.
\end{proof}

\subsection{Comparison of Theorem~\ref{thmCanonBound} to the cell model lower bound}
\label{sec:compare}

Given a lattice packing of $B_n$ with $k = c_1 d \cdot 2^{-d} (1-\eps)^d n$ spheres of radius $r_d/(1-\eps)$ (and thus density $\Theta(d \cdot 2^{-d})$),
construct the Voronoi diagram around the centers of the packing.  Around each center, place  a copy of its Voronoi cell scaled down by a factor $\eps$.  If the centers are allowed to move arbitrarily within their respective shrunken cells, they still form a packing of spheres of radius $r_d$. 
 The density of such a packing is $c_1 (1-\eps)^d  d 2^{-d}$, and  so if we take $\eps = c_2/d$, then the resulting packing still has density $\sim c_1 e^{-c_2} d\cdot 2^{-d}$.  The probability that a random set of $k$ points in $B_n$ is such a configuration  is the probability that each of the $k$ shrunken cells contain exactly one of $k$ uniformly random points, that is:
 \begin{align*}
  \frac{k!}{n^k} \eps^{dk} (n/k)^k
\end{align*}
since the volume of each shrunken Voronoi cell is $n/k \cdot \eps^d$. This gives 
\begin{align*}
\hat Z_{B_n}(k) \ge  \eps^{dk} (n/k)^k \, ,
\end{align*}
 and so with $\alpha = k/n \sim c_1 e^{-c_2} d\cdot 2^{-d}$,
\begin{align*}
\frac{1}{\alpha n} \log \frac{\hat Z_{B_n}(k)}{n^k/k!} \ge \frac{1}{\alpha n} \log \frac{\eps^{dk}}{e^k} = - (1+o_d(1)) d \log d \,,
\end{align*}
which is considerably smaller (of a different asymptotic order) than the bound in Theorem~\ref{thmCanonBound}.

\section*{Acknowledgements}

We thank Yoav Kallus and Cris Moore for inspiring discussions at the Santa Fe Institute and Tyler Helmuth and Francesco Zamponi for many helpful comments on the manuscript.


\begin{thebibliography}{10}

\bibitem{ajtai1980note}
M.~Ajtai, J.~Koml\'os, and E.~Szemer\'edi.
\newblock A note on {R}amsey numbers.
\newblock {\em J. Combin. Theory Ser. A}, 29:354--360, 1980.

\bibitem{ball1992lower}
K.~Ball.
\newblock A lower bound for the optimal density of lattice packings.
\newblock {\em Internat. Math. Res. Notices}, 1992(10):217--221, 1992.

\bibitem{GlassJamming}
P.~Charbonneau, J.~Kurchan, G.~Parisi, P.~Urbani, and F.~Zamponi.
\newblock Glass and jamming transitions: From exact results to finite-dimensional descriptions.
\newblock {\em Annual Review of Condensed Matter Physics}, 8: 265--288, 2017.


\bibitem{cohn2016packing}
H.~Cohn.
\newblock Packing, coding, and ground states.
\newblock {\em arXiv preprint arXiv:1603.05202}, 2016.

\bibitem{cohn2016conceptual}
H.~Cohn.
\newblock A conceptual breakthrough in sphere packing.
\newblock {\em Notices Amer. Math. Soc.}, 64:102--115, 2017.

\bibitem{cohn2016sphere}
H.~Cohn, A.~Kumar, S.~D. Miller, D.~Radchenko, and M.~Viazovska.
\newblock The sphere packing problem in dimension 24.
\newblock {\em Ann. of Math. (2)}, 185:1017--1033, 2017.

\bibitem{cohn2014sphere}
H.~Cohn and Y.~Zhao.
\newblock Sphere packing bounds via spherical codes.
\newblock {\em Duke Math. J.}, 163:1965--2002, 2014.

\bibitem{conway2013sphere}
J.~H. Conway and N.~J.~A. Sloane.
\newblock {\em Sphere packings, lattices and groups}, volume 290.
\newblock Springer, third edition, 1999.

\bibitem{davenport1947hlawka}
H.~Davenport and C.~A. Rogers.
\newblock Hlawka's theorem in the geometry of numbers.
\newblock {\em Duke Math. J.}, 14:367--375, 1947.

\bibitem{Davies2015}
E.~Davies, M.~Jenssen, W.~Perkins, and B.~Roberts.
\newblock Independent sets, matchings, and occupancy fractions.
\newblock {\em Journal of the London Mathematical Society}, to appear.

\bibitem{davies2016average}
E.~Davies, M.~Jenssen, W.~Perkins, and B.~Roberts.
\newblock On the average size of independent sets in triangle-free graphs.
\newblock {\em Proceedings of the American Mathematical Society}, to appear.

\bibitem{diaconis2011geometric}
P.~Diaconis, G.~Lebeau, and L.~Michel.
\newblock Geometric analysis for the {M}etropolis algorithm on {L}ipschitz
  domains.
\newblock {\em Invent. Math.}, 185:239--281, 2011.

\bibitem{fernandez2007analyticity}
R.~Fern\'andez, A.~Procacci, and B.~Scoppola.
\newblock The analyticity region of the hard sphere gas. {I}mproved bounds.
\newblock {\em J. Stat. Phys.}, 5:1139--1143, 2007.

\bibitem{frisch}
H.~L. Frisch and J.~K. Percus.
\newblock High dimensionality as an organizing device for classical fluids.
\newblock {\em Physical Review E}, 60(3): 2942--2948, 1999.


\bibitem{hales2005proof}
T.~C. Hales.
\newblock A proof of the {K}epler conjecture.
\newblock {\em Ann. of Math. (2)}, 162(3):1065--1185, 2005.

\bibitem{hayes2014lower}
T.~P. Hayes and C.~Moore.
\newblock Lower bounds on the critical density in the hard disk model via
  optimized metrics.
\newblock {\em arXiv preprint arXiv:1407.1930}, 2014.

\bibitem{hofer2015disagreement}
C.~Hofer-Temmel.
\newblock Disagreement percolation for the hard-sphere model.
\newblock {\em arXiv preprint arXiv:1507.02521}, 2015.

\bibitem{JJP18kissing}
M.~Jenssen, F.~Joos, and W.~Perkins.
\newblock On kissing numbers and spherical codes in high dimensions.
\newblock {\em arXiv preprint arXiv:1803.02702}, 2018.

\bibitem{jiang2004asymptotic}
T.~Jiang and A.~Vardy, \emph{Asymptotic improvement of the
  {G}ilbert-{V}arshamov bound on the size of binary codes}, IEEE Trans. Inform.
  Theory \textbf{50} (2004), 1655--1664.

\bibitem{kabatiansky1978bounds}
G.~A. {Kabatjanski\u\i} and V.~I. {Leven\v ste\u\i n}.
\newblock Bounds for packings on the sphere and in space.
\newblock {\em Problemy Pereda\v ci Informacii}, 14:3--25, 1978.

\bibitem{kannan2003rapid}
R.~Kannan, M.~W. Mahoney, and R.~Montenegro.
\newblock Rapid mixing of several {M}arkov chains for a hard-core model.
\newblock In {\em Algorithms and computation}, volume 2906 of {\em Lecture
  Notes in Comput. Sci.}, pages 663--675. Springer, Berlin, 2003.

\bibitem{krivelevich2004lower}
M.~Krivelevich, S.~Litsyn, and A.~Vardy.
\newblock A lower bound on the density of sphere packings via graph theory.
\newblock {\em Int. Math. Res. Not.}, 2004(43):2271--2279, 2004.

\bibitem{lowen2000fun}
H.~L{\"o}wen.
\newblock Fun with hard spheres.
\newblock In {\em Statistical physics and spatial statistics}, pages 295--331.
  Springer, 2000.

\bibitem{mase2001packing}
S.~Mase, J.~M{\o}ller, D.~Stoyan, R.~P. Waagepetersen, and G.~D{\"o}ge.
\newblock Packing densities and simulated tempering for hard core {G}ibbs point
  processes.
\newblock {\em Ann. Inst. Statist. Math.}, 53(4):661--680, 2001.


\bibitem{Parisi}
 G.~Parisi and F.~Zamponi.
\newblock Mean-field theory of hard sphere glasses and jamming.
\newblock {\em Rev. Mod. Phys.}, 82(1):789--845, 2010.


\bibitem{pulvirenti2012cluster}
E.~Pulvirenti and D.~Tsagkarogiannis.
\newblock Cluster expansion in the canonical ensemble.
\newblock {\em Comm. Math. Phys.}, 316(2):289--306, 2012.

\bibitem{radin2005structure}
C.~Radin and L.~Sadun.
\newblock Structure of the hard sphere solid.
\newblock {\em Physical review letters}, 94(1):015502, 2005.

\bibitem{richthammer2007translation}
 T.~Richthammer.
\newblock Translation-invariance of two-dimensional Gibbsian point processes.
\newblock {\em Comm. Math. Phys.}, 274(1):81--122, 2007.


\bibitem{rogers1947existence}
C.~A. Rogers.
\newblock Existence theorems in the geometry of numbers.
\newblock {\em Ann. of Math.}, 48:994--1002, 1947.

\bibitem{rogers1964packing}
C.~A. Rogers.
\newblock {\em Packing and covering}.
\newblock Cambridge University Press, 1964.

\bibitem{shearer1983note}
J.~B. Shearer.
\newblock A note on the independence number of triangle-free graphs.
\newblock {\em Discrete Math.}, 46:83--87, 1983.

\bibitem{siegel1945mean}
C.~L. Siegel.
\newblock A mean value theorem in geometry of numbers.
\newblock {\em Ann. of Math. (2)}, 46:340--347, 1945.

\bibitem{torquato2006new}
S.~Torquato and F.~H. Stillinger.
\newblock New conjectural lower bounds on the optimal density of sphere
  packings.
\newblock {\em Experiment. Math.}, 15(3):307--331, 2006.

\bibitem{vance2011improved}
S.~Vance.
\newblock Improved sphere packing lower bounds from {H}urwitz lattices.
\newblock {\em Adv. Math.}, 227:2144--2156, 2011.

\bibitem{venkatesh2012note}
A.~Venkatesh.
\newblock A note on sphere packings in high dimension.
\newblock {\em Int. Math. Res. Not.}, 2013(7):1628--1642, 2013.

\bibitem{viazovska2017sphere}
M.~S. Viazovska.
\newblock The sphere packing problem in dimension 8.
\newblock {\em Ann. of Math. (2)}, 185:991--1015, 2017.

\end{thebibliography}
\end{document}